\documentclass[11pt, a4paper]{amsart}

\usepackage{amssymb,amsmath}
\usepackage{mathrsfs}
\usepackage{mathabx}
\usepackage{mathtools}
\usepackage{tikz-cd}
\usetikzlibrary{arrows}

\usepackage{amsthm}
\usepackage{thmtools}

\usepackage{graphicx}
\usepackage{microtype}

\usepackage[english]{babel}

\usepackage[hidelinks]{hyperref}
\usepackage{cite}
\usepackage[capitalize]{cleveref}

\newtheorem{theorem}{Theorem}[section]
\newtheorem{lemma}[theorem]{Lemma}
\newtheorem{corollary}[theorem]{Corollary}
\newtheorem{proposition}[theorem]{Proposition}

\theoremstyle{definition}
\newtheorem{definition}[theorem]{Definition}
\newtheorem{remark}[theorem]{Remark}
\newtheorem{question}[theorem]{Question}

\newcommand{\isom}{\cong}

\newcommand{\F}{\mathbb{F}}
\newcommand{\Z}{\mathbb{Z}}

\makeatletter
\def\Ddots{\mathinner{\mkern1mu\raise\p@
\vbox{\kern7\p@\hbox{.}}\mkern2mu
\raise4\p@\hbox{.}\mkern2mu\raise7\p@\hbox{.}\mkern1mu}}

\makeatother

\newcommand{\normal}[1]{\langle\!\langle #1 \rangle\!\rangle}

\def\injects{\hookrightarrow}

\DeclareSymbolFontAlphabet{\amsmathbb}{AMSb}

\DeclareMathOperator{\rk}{rk}
\DeclareMathOperator{\rr}{rr}

\DeclareMathOperator{\core}{Core}

\DeclareMathOperator{\aut}{Aut}

\DeclareMathOperator{\CAT}{CAT}
\DeclareMathOperator{\id}{Id}

\DeclareMathOperator{\sym}{Sym}

\DeclarePairedDelimiter\abs{\lvert}{\rvert}

\makeatletter
\let\oldabs\abs
\def\abs{\@ifstar{\oldabs}{\oldabs*}}

\newcounter{cases}
\newcounter{subcases}[cases]

\tikzset{
math to/.tip={Glyph[glyph math command=rightarrow]},
loop/.tip={Glyph[glyph math command=looparrowleft, swap]},
loop'/.tip={Glyph[glyph math command=looparrowleft]},
 weird/.tip={Glyph[glyph math command=Rrightarrow, glyph length=1.5ex]},
  pi/.tip={Glyph[glyph math command=pi, glyph length=1.5ex, glyph axis=0pt]},
}

\begin{document}

\title[Embedding fg free-by-cyclic groups in \{fg free\}-by-cyclic groups]{Embedding finitely generated free-by-cyclic groups in \{finitely generated free\}-by-cyclic groups}

\author{Marco Linton}
\address{Instituto de Ciencias Matem\'aticas, CSIC-UAM-UC3M-UCM, Madrid, Spain}
\email{marco.linton@icmat.es}

\begin{abstract}
We refine Feighn--Handel's results on subgroups of mapping tori of free groups to the special case of free-by-cyclic groups. We use these refinements to show that any finitely generated free-by-cyclic group embeds in a \{finitely generated free\}-by-cyclic group as a retract. When the free-by-cyclic group is hyperbolic, it embeds in a hyperbolic \{finitely generated free\}-by-cyclic group as a quasi-convex subgroup. Combined with a result of Hagen--Wise, this implies that all hyperbolic free-by-cyclic groups are cocompactly cubulated.
\end{abstract}

\maketitle

\section{Introduction}

Feighn--Handel developed a powerful tool in \cite{FH99} to study finitely generated subgroups of mapping tori $M(\psi)$ of free group monomorphisms $\psi\colon \F\to \F$. The purpose of this article is to refine Feighn--Handel's techniques for the special case in which $\psi$ is an automorphism ---that is, when $M(\psi)$ is a free-by-cyclic group $\F\rtimes_{\psi}\Z$. Here and throughout the article, we do not assume that $\F$ is finitely generated.

Feighn--Handel's main technical result stated that a finitely generated mapping torus of a free group $M(\psi)$ splits as a HNN-extension $F*_{\phi}$ with $F$ a finitely generated free group and with $\phi$ an isomorphism identifying a free factor of $F$ with another subgroup of $F$. Our first main result, \cref{FHmain2}, shows that a finitely generated free-by-cyclic group $\F\rtimes_{\psi}\Z$ splits as a HNN-extension $F*_{\phi}$ with $F$ a finitely generated free factor of $\F$ and with $\phi$ an isomorphism identifying two free factors of $F$. 

Using this, we characterise when a free-by-cyclic group $\F\rtimes_{\psi}\Z$ is finitely generated in terms of the automorphism $\psi$. Note that every non-trivial subgroup of a free-by-$\Z$ group is itself free-by-$\Z$.

\begin{theorem}
\label{thm:main1}
Let $\F$ be a free group, let $\psi\in\aut(\F)$ be an automorphism and let $G = \F\rtimes_{\psi}\Z$. Then $G$ is finitely generated if and only if there is a free product decomposition
\[
\F = A*\left(\Asterisk_{i\in \Z}C_i\right)
\]
where $A$ and $C_0$ are finitely generated and where $C_i = \psi^i(C_0)$ for all $i\in \Z$.
\end{theorem}

In \cite{Li25b} the author proved that finitely generated mapping tori of free groups have a canonical collection of maximal submapping tori with respect to which they are relatively hyperbolic and relatively locally quasi-convex. Using \cref{thm:main1}, we state a strengthening of this result, \cref{thm:rel_hyp}.

Chong--Wise proved in \cite{CW24} that a finitely generated mapping torus of a free group embeds in a mapping torus of a finitely generated free group. They conjectured \cite[Conjecture 1.2]{CW24} that an analogous statement should hold for free-by-cyclic groups: that is, a finitely generated free-by-cyclic group embeds as a subgroup of a \{finitely generated free\}-by-cyclic group. Our second main result uses \cref{thm:main1} to confirm their conjecture. A more general question, which we also answer, was also posed independently by Wu--Ye in \cite[Question 7.4]{WY25}.

\begin{theorem}
\label{thm:main2}
If $G = \F\rtimes_{\psi}\Z$ is a finitely generated free-by-$\Z$ group, then $G$ embeds in a \{finitely generated free\}-by-$\Z$ group. More explicitly, there is some finitely generated free group $F$, an automorphism $\theta\in \aut(F)$ and an embedding $\iota\colon\F\injects F$ making the following diagram commute:
\[
\begin{tikzcd}
\mathbb{F} \arrow[d, "\iota", hook] \arrow[r, "\psi"] & \mathbb{F} \arrow[d, "\iota", hook] \\
F \arrow[r, "\theta"]                                 & F                                  
\end{tikzcd}
\]
Moreover, denoting by $H = F\rtimes_{\theta}\Z$, the pair $(H, \{G\})$ is relatively hyperbolic and $G$ is a retract of $H$.
\end{theorem}

Bridson--Groves showed that a \{finitely generated free\}-by-$\Z$ group satisfies a quadratic isoperimetric inequality \cite{BG10}. The fact that a finitely generated free-by-cyclic group $G$ embeds as a retract of such a group then implies that $G$ will also satisfy a quadratic isoperimetric inequality.

\begin{corollary}
\label{cor:iso}
A finitely generated free-by-$\Z$ group satisfies a quadratic isoperimetric inequality.
\end{corollary}

Since a group that is relatively hyperbolic with respect to hyperbolic groups is hyperbolic by a result of Osin \cite[Corollary 2.41]{Os06}, we obtain the following corollary.

\begin{corollary}
A finitely generated hyperbolic free-by-$\Z$ group embeds as a quasi-convex subgroup of a hyperbolic \{finitely generated free\}-by-$\Z$ group.
\end{corollary}

Hagen--Wise showed in \cite{HW15} that a hyperbolic \{fg free\}-by-$\Z$ group acts properly and cocompactly on a $\CAT(0)$ cube complex. By a result of Agol \cite{Ag13}, this implies that hyperbolic \{fg free\}-by-$\Z$ groups are virtually compact special in the sense of Haglund--Wise \cite{HW08}. Since compact special groups are precisely the quasi-convex subgroups of right angled Artin groups (see \cite{HW08}), it follows that quasi-convex subgroups of hyperbolic virtually compact special groups are themselves virtually compact special. Combining these facts with \cref{thm:main2} we obtain the following corollary.

\begin{corollary}
\label{cor}
A hyperbolic free-by-cyclic group is virtually compact special.
\end{corollary}

Since every finitely generated subgroup of a hyperbolic free-by-cyclic group is hyperbolic by a result of Gersten \cite{Ge96} (combined with coherence \cite{FH99}), \cref{cor} implies that every finitely generated subgroup of a hyperbolic free-by-cyclic group is also virtually compact special. 

Cohen proved a similar statement in \cite{Co25} for one-relator groups: a one-relator group admitting an acylindrical one-relator hierarchy (which is hence hyperbolic and virtually compact special by \cite{Li25a}) has all its finitely generated subgroups virtually compact special. Via the results in \cite{KL24a}, \cref{cor} implies that all hyperbolic virtually compact special one-relator groups also have have all their finitely generated subgroups virtually compact special.

An automorphism $\psi\in \aut(\F)$ is \emph{fully irreducible} if there is no finitely generated non-trivial proper free factor $H\leqslant \F$ and integer $m\geqslant 1$ so that $\psi^m(H)$ is conjugate to $H$. The automorphism $\theta$ from \cref{thm:main2} is never fully irreducible by its construction. It is likely that \cref{thm:main2} can be improved when $\psi$ is assumed fully irreducible.

\begin{question}
\label{question}
If $\psi\in \aut(\F)$ is a fully irreducible automorphism so that $\F\rtimes_{\psi}\Z$ is finitely generated, is there a finitely generated free group $F$, an embedding $\iota\colon \F\to F$ and a fully irreducible automorphism $\theta\in \aut(F)$ so that $\theta\circ\iota = \iota\circ\psi$?
\end{question}

The analogue of \cref{question} for monomorphisms $\F\to\F$ was established by Chong--Wise \cite{CW24}.

\subsection*{Acknowledgements}
The author thanks Ashot Minasyan for suggesting proving that $G$ can embedded as a retract in \cref{thm:main2} and for pointing out a simplification in the proof of \cref{thm:main2}. The author also thanks the anonymous referee for helpful comments and Martin Bridson for pointing out \cref{cor:iso}. This work has received support from the grant 202450E223 (Impulso de líneas científicas estratégicas de ICMAT).

\section{Preliminaries}

\subsection{Graphs, folds and subgroups of free groups}

A \emph{graph} is a 1-dimensional CW-complex $\Gamma$. The 0-cells of $\Gamma$ are the \emph{vertices} and the 1-cells are the \emph{edges}. We shall use $I$ to denote any graph homeomorphic to the interval and $S^1$ to denote any graph homeomorphic to the circle. Maps $\Gamma\to\Lambda$ between graphs will always be assumed to be combinatorial ---that is, vertices map to vertices and edges map homeomorphically to edges. Maps from the (subdivided) interval $I\to \Gamma$ are called paths and maps from the (subdivided) circle $S^1\to \Gamma$ are cycles. A loop is simply a path with the same initial and terminal vertex. A map of graphs is an \emph{immersion} if it is locally injective. A pointed graph is a pair $(\Gamma, v)$ with $\Gamma$ a graph and $v\in \Gamma$ a vertex. A pointed map $(\Gamma, v)\to (\Lambda, u)$ is a map of graphs $\Gamma\to\Lambda$ sending the basepoint $v$ to the basepoint $u$.

If $f\colon (\Gamma, v)\to(\Lambda, u)$ is a graph map, write $f_*\colon \pi_1(\Gamma, v)\to \pi_1(\Lambda, u)$ for the induced map on fundamental groups. When $f$ is understood, we use $\Gamma^{\#}$ to denote the image $f_*(\pi_1(\Gamma, v))$.

A map of graphs $f\colon \Gamma\to \Lambda$ is a \emph{fold} if it identifies a single pair of oriented edges $e_1, e_2\subset \Gamma$ with a common origin. It is a classical observation of Stallings' \cite{St83} that any graph map $f\colon \Gamma\to \Lambda$, with $\Gamma$ a finite graph, factors as a finite sequence of folds $f_i\colon \Gamma_{i-1}\to\Gamma_i$ followed by an immersion $h\colon \Gamma_k\to \Lambda$:
\[
\begin{tikzcd}
\Gamma = \Gamma_0 \arrow[r, "f_1"'] \arrow[rrrr, "f", bend left] & \Gamma_1 \arrow[r, "f_2"'] & \ldots \arrow[r, "f_k"'] & \Gamma_k \arrow[r, "h"'] & \Lambda
\end{tikzcd}
\]
Although the sequence of folds $f_i$ may not be unique, the graph $\Gamma_k$ and the map $h$ is unique.

A graph $\Gamma$ is \emph{core} if it is the union of the images of immersed cycles $S^1\to \Gamma$. A pointed graph $(\Gamma, v)$ is \emph{pointed core} if it is the union of the images of immersed loops $I\to \Gamma$ at $v$. If $\Gamma$ is a graph, the \emph{core of $\Gamma$} is the subgraph $\core(\Gamma)\subset \Gamma$ consisting of the union of all images of immersed cycles in $\Gamma$ (it may be empty). If $(\Gamma, v)$ is a pointed graph, the \emph{pointed core of $(\Gamma, v)$} is the pointed subgraph $\core(\Gamma, v)\subset\Gamma$ consisting of the union of all images of immersed loops at $v$.

If $H\leqslant \pi_1(\Lambda, u)$ is any subgroup, then there is a unique immersion $f\colon (\Gamma, v)\to(\Lambda, u)$ so that $(\Gamma, v)$ is pointed core and $f_*(\pi_1(\Gamma, v)) = H$. This (pointed) graph, which we shall denote by $\Gamma(H)$ and call the \emph{pointed subgroup graph for $H$}, is the pointed core subgraph of the pointed cover of $\Lambda$ associated to $H$. 

If $[H]$ is the conjugacy class of a subgroup $H\leqslant \pi_1(\Lambda, u)$, then there is a unique immersion $f\colon \Gamma\to\Lambda$ so that $\Gamma$ is core and $[f_*(\pi_1(\Gamma, v))] = [H]$, where $v\in f^{-1}(u)$. This graph, which we shall denote by $\Gamma[H]$ and call the \emph{subgroup graph for $[H]$}, is the core subgraph of the cover of $\Lambda$ associated to $H$.

If $\delta\in \pi_1(\Lambda, u)$, we shall often write $\Gamma(\delta)$ or $\Gamma[\delta]$ instead of $\Gamma(\langle \delta\rangle)$ or $\Gamma[\langle\delta\rangle]$.

\subsection{Graph pairs and graph triples}

Let $f_Z\colon (Z, v)\to (R, u)$ be a pointed map of connected graphs. If $X\subset Z$ is a connected subgraph containing the basepoint, we call $(Z, X)$, together with the map $f_Z$, which we often suppress, a \emph{graph pair}. If $Y, X\subset Z$ are connected subgraphs such that $X\cap Y$ is connected and contains the basepoint and $X\cup Y = Z$, we call $(Z, Y, X)$, together with the map $f_Z$, which we often suppress, a \emph{graph triple}. Note that these conditions imply that $\pi_1(Z, v)$ is generated by $\pi_1(X, v)$ and $\pi_1(Y, v)$. A graph pair $(Z, X)$ or a graph triple $(Z, Y, X)$ is \emph{tight} if the map $f_Z$ is an immersion.

The \emph{relative rank} of a graph pair $(Z, X)$ is defined as
\[
\rr(Z, X) = \rk(\pi_1(Z, v)/\normal{\pi_1(X, v)}),
\]
where $\rk(G)$ denotes the \emph{rank} of the group $G$. Since $X\subset Z$, the subgroup $\pi_1(X, v)$ is a free factor of $\pi_1(Z, v)$ and so the quotient $\pi_1(Z, v)/\normal{\pi_1(X, v)}$ is a free group. When $Z - X$ consists of finitely many 0-cells and 1-cells, we have
\[
\rr(Z, X) = \rk(\pi_1(Z, v)) - \rk(\pi_1(X, v)) = -\chi(Z, X).
\]

Let $(Z_1, X_1)$ and $(Z_2, X_2)$ be graph pairs with $f_1\colon Z_1\to R$, $f_2\colon Z_2\to R$ the associated maps. A \emph{map of graph pairs} is a pair
\[
q = (q_{Z}, q_X) \colon (Z_1, X_1)\to (Z_2, X_2)
\]
of graph maps $q_Z\colon Z_1\to Z_2$, $q_X\colon X_1\to X_2$ making the following diagram commute:
\[
\begin{tikzcd}
X_1 \arrow[r, hook] \arrow[d, "q_X"] & Z_1 \arrow[rd, "f_1"] \arrow[d, "q_Z"] &   \\
X_2 \arrow[r, hook]                  & Z_2 \arrow[r, "f_2"]                   & R
\end{tikzcd}
\]

Similarly, if $(Z_1, Y_1, X_2)$ and $(Z_2, Y_2, X_2)$ are graph triples, a \emph{map of graph triples} is a triple
\[
q = (q_Z, q_Y, q_X)\colon (Z_1, Y_1, X_1)\to (Z_2, Y_2, X_2)
\]
of graph maps $q_Z\colon Z_1\to Z_2$, $q_Y\colon Y_1\to Y_2$, $q_X\colon X_1\to X_2$ making the appropriate diagram commute. 

We shall usually assume that our graph pairs and triples and maps of graph pairs and triples are pointed. For graph pairs $(Z, X)$, the basepoint is required to be in $X$. For graph triples $(Z, Y, X)$, the basepoint is required to be in $X\cap Y$.

\begin{remark}
Note that if $(Z, Y, X)$ is a graph triple, then $(Y, X\cap Y)$ and $(X, X\cap Y)$, together with the induced maps, are both graph pairs. 
\end{remark}

\subsection{Invariant graph pairs}

The main objects of study in Feighn--Handel's article \cite{FH99} were ($\psi$-)invariant graph pairs, which we now define.

Let $\F$ be a free group, let $\psi\colon \F\to \F$ be a monomorphism and assume we have identified $\F$ with $\pi_1(R, u)$ for some pointed graph $(R, u)$. Recall that the \emph{mapping torus of $\psi$} is the group:
\[
M(\psi) = \langle \F, t \mid t^{-1}ft = \psi(f), \, f\in \F\rangle.
\]

Let $(Z, X)$ be a graph pair with associated graph map $f_Z\colon Z\to R$. Say the pair $(Z, X)$ is \emph{$\psi$-invariant} (or just invariant) if
\[
Z^{\#} = \langle X^{\#}, \psi(X^{\#})\rangle.
\]
If $H\leqslant M(\psi)$ is a subgroup with $t\in H$, say $(Z, X)$ is an \emph{$\psi$-invariant graph pair for $H$} if $(Z, X)$ is a $\psi$-invariant graph pair and if
\[
H = \langle X^{\#}, t\rangle.
\]

If $H\leqslant M(\psi)$ is a finitely generated subgroup with $t\in H$, say a finite graph pair $(Z, X)$ is \emph{minimal for $H$} if $(Z, X)$ is an invariant graph pair for $H$ and if 
\[
\rr(Z, X)\leqslant \rr(Z', X')
\]
for all finite invariant graph pairs $(Z', X')$ for $H$.

The main technical result in \cite{FH99} states that a finitely generated subgroup $H$ of a mapping torus (that contains $t$) has a particularly nice HNN-extension decomposition which can be read off from a minimal invariant graph pair for $H$. See \cite[Theorem 4.6]{Li25b} for the version stated below.

\begin{theorem}
\label{FHmain}
Let $(Z, X)$ be a finite $\psi$-invariant graph pair for $H$ with $(f_Z)_*$ injective and let $C\leqslant Z^{\#}$ be a free factor so that $Z^{\#} = X^{\#}*C$. The following are equivalent:
\begin{enumerate}
\item\label{itm:min} The pair $(Z, X)$ is minimal.
\item\label{itm:inj} The map
\[
\theta_n\colon\pi_1\left(X\vee \bigvee_{i=0}^n\Gamma(\psi^i(C)), v_Z\right) \to \pi_1(R, v_R)
\]
is injective for all $n\geqslant 0$.
\item\label{itm:pres} We have
\[
H\isom \langle Z^{\#}, t \mid t^{-1}xt = \psi(x), \forall x\in X^{\#}\rangle.
\]
\end{enumerate}
In particular, if any of the above hold, then $\chi(H) = -\rr(Z, X)$.
\end{theorem}

\subsection{Bi-invariant graph triples}
\label{sec:triples}

We now define the main objects of study of this article, bi-invariant graph triples.

Let $\F$ be a free group, let $\psi\colon \F\to \F$ be an isomorphism and assume we have identified $\F$ with $\pi_1(R, u)$ for some pointed graph $(R, u)$. Then the mapping torus of $\psi$ is the free-by-cyclic group:
\[
M(\psi) \cong \F\rtimes_{\psi}\Z.
\]

Let $(Z, Y, X)$ be a graph triple with associated map $f_Z\colon Z\to R$. Say the triple $(Z, Y, X)$ is \emph{$\psi$-bi-invariant} (or just bi-invariant) if $(Y, X\cap Y)$ is $\psi$-invariant and $(X, X\cap Y)$ is $\psi^{-1}$-invariant. Note that $(Y, X\cap Y)$ being $\psi$-invariant implies that $(Z, X)$ is $\psi$-invariant. Similarly, $(X, X\cap Y)$ being $\psi^{-1}$-invariant implies that $(Z, Y)$ is $\psi^{-1}$-invariant.

\begin{lemma}
\label{iso}
If $(Z, Y, X)$ is a $\psi$-bi-invariant graph triple, then the induced homomorphism
\[
\psi\mid_{X^{\#}}\colon X^{\#} \to Y^{\#}
\]
is an isomorphism.
\end{lemma}

\begin{proof}
We have that $(Y, X\cap Y)$ is a $\psi$-invariant graph pair and $(X, X\cap Y)$ is a $\psi^{-1}$-invariant graph pair. In particular,
\begin{align*}
Y^{\#} &= \langle (X\cap Y)^{\#}, \psi((X\cap Y)^{\#})\rangle\\
X^{\#} &= \langle (X\cap Y)^{\#}, \psi^{-1}((X\cap Y)^{\#})\rangle.
\end{align*}
Thus, $\psi(X^{\#}) = Y^{\#}$ and $\psi^{-1}(Y^{\#}) = X^{\#}$. Since $\psi$ is an automorphism of $\F$, it follows that $\psi\mid_{X^{\#}}$ is an isomorphism as claimed.
\end{proof}

If $H\leqslant M(\psi)$ is a subgroup with $t\in H$, say $(Z, Y, X)$ is a \emph{$\psi$-bi-invariant graph triple for $H$} if $(Z, Y, X)$ is $\psi$-bi-invariant and if 
\[
H = \langle (X\cap Y)^{\#}, t\rangle. 
\]
Since $Z^{\#} = \langle X^{\#}, Y^{\#}\rangle$, this is equivalent to $(Z, Y, X)$ being $\psi$-bi-invariant and:
\[
\left\langle \bigcup_{i\in\Z} \psi^i(Z^{\#})\right\rangle = H\cap \F.
\]

If $H\leqslant M(\psi)$ is a finitely generated subgroup with $t\in H$, say a finite graph triple $(Z, Y, X)$ is \emph{minimal for $H$} if $(Z, Y, X)$ is a bi-invariant graph triple for $H$ and if 
\[
\rr(Z, X) + \rr(Z, Y) \leqslant \rr(Z', X') + \rr(Z', Y')
\]
for all finite bi-invariant graph triples $(Z', Y', X')$ for $H$. We shall see in \cref{FHmain2} that if $(Z, Y, X)$ is a finite minimal graph triple for $H$, then $(Z, X)$ and $(Z, Y)$ are minimal graph pairs for $H$.

\begin{remark}
If $H\leqslant M(\psi)$ is a subgroup with $t\in H$, there is a subgroup $L\leqslant \F$ so that $H = \langle L, t\rangle$. In particular, the graph triple
\[
(\Gamma(L)\vee\Gamma(\psi(L))\vee\Gamma(\psi^{-1}(L)), \Gamma(L)\vee\Gamma(\psi(L)), \Gamma(L)\vee\Gamma(\psi^{-1}(L))
\]
is a bi-invariant graph triple for $H$. If $H$ is finitely generated, then $L$ can be taken to be finitely generated so that the graph triple above will be finite.
\end{remark}

In \cref{sec:presentations} we shall prove the analogue of \cref{FHmain} for graph triples.

\section{Tightening graph triples}

We fix the same notation as in \cref{sec:triples}. Namely, $\F$ is a free group, $\psi\colon \F\to\F$ is an isomorphism, $(R, u)$ is a pointed graph and there is an identification $\F\cong \pi_1(R, u)$. The aim of this section is to define a version of Feighn--Handel's tightening procedure for graph triples and prove \cref{prop:injective}.

We begin by identifying the various types of folds for a graph triple. There will be more cases than there were for the graph pair case.

\subsection{Types of folds}

Let $(Z, Y, X)$ be a $\psi$-bi-invariant graph triple for a subgroup $H$ with basepoint $v_Z\in X\cap Y$. Let $e_1, e_2\subset Z$ be two oriented edges with common origin which map to the same edge in $R$. Denote by $Z_1$ the graph obtained from $Z$ by folding $e_1$ and $e_2$ and denote by $q\colon Z\to Z_1$ the folding map. Denote by $v, w_1$ and $v, w_2$ the two endpoints of $e_1, e_2$ respectively.

Denote also by $\check{q}_{X}\colon X\to q(X) = X_1$, $\check{q}_Y\colon Y\to q(Y) = Y_1$ and by $\check{q}_{X\cap Y}\colon X\cap Y\to q(X\cap Y)$ the induced maps on $X, Y$ and $X\cap Y$. Just like in \cite{FH99} we use the different notation for these maps to emphasise that $\check{q}_{X}, \check{q}_Y$ and $\check{q}_{X\cap Y}$ may not be folds, even though $q$ is. Moreover, we may also have that $q(X\cap Y)\neq X_1\cap Y_1$.

We will maintain this notation for the remainder of this section.

There are several cases to consider for $q$:
\begin{enumerate}
\item $q$ is an \emph{ordinary fold} if $\check{q}_{X}, \check{q}_{Y}$ are folds or homeomorphisms. If $q$ is ordinary, then say:
\begin{enumerate}
\item $q$ is an \emph{ordinary $X$-fold} if $\check{q}_X$ is a fold.
\item $q$ is an \emph{ordinary $Y$-fold} if $\check{q}_Y$ is a fold.
\end{enumerate}
\item $q$ is a \emph{single exceptional fold} if only one of $\check{q}_{X}, \check{q}_{Y}$ is a fold or a homeomorphism. If $q$ is a single exceptional fold, then say:
\begin{enumerate}
\item $q$ is an \emph{$X$-exceptional fold} if $\check{q}_X$ is a fold or a homeomorphism.
\item $q$ is a \emph{$Y$-exceptional fold} if $\check{q}_Y$ is a fold or a homeomorphism.
\end{enumerate}
\item $q$ is a \emph{double exceptional fold} if neither $\check{q}_{X}$ nor $\check{q}_{Y}$ is a fold or a homeomorphism.
\end{enumerate}

Note that an ordinary fold does not have to be an ordinary $X$-fold or an ordinary $Y$-fold as it can be a homeomorphism on both $X$ and $Y$.

We immediately note the following.

\begin{lemma}
\label{rr_inequality}
We have:
\begin{align*}
\rr(Z_1, X_1) &\leqslant \rr(Z, X)\\
\rr(Z_1, Y_1) &\leqslant \rr(Z, Y).
\end{align*}
\end{lemma}

\begin{proof}
When $Z$ is finite, this is \cite[Lemma 4.5 (4)]{FH99} applied to the $\psi$-invariant graph pair $(Z, X)$ and the $\psi^{-1}$-invariant graph pair $(Z, Y)$. The general case follows by the same reasoning.
\end{proof}

Before analysing the three types of folds, we shall need a general criterion for when $q$ is not a fold or a homeomorphism on a chosen subgraph. The following is immediate from the definitions.

\begin{lemma}
\label{not_fold}
If $A\subset Z$ is a subgraph, then $q\mid_{A}\colon A\to q(A)$ is not a fold or a homeomorphism if and only if the following three conditions hold:
\begin{itemize}
\item $w_1, w_2\in A$.
\item $w_1\neq w_2$.
\item $e_1\not\subset A$ or $e_2\not\subset A$.
\end{itemize}
In particular, if $q\mid_{A}$ is not a fold or a homeomorphism, then it identifies a pair of vertices.
\end{lemma}

We now analyse the three cases separately.

\subsection{Ordinary folds}

The case in which $q$ is an ordinary fold is the simplest to handle.

\begin{lemma}
\label{ordinary_fold}
If $q$ is an ordinary fold, then one of the following holds:
\begin{enumerate}
\item $q$ is an ordinary $X$-fold or an ordinary $Y$-fold. In this case $\check{q}_{X\cap Y}$ is a fold or a homeomorphism and $X_1\cap Y_1 = q(X\cap Y)$.
\item\label{itm:ordinary_case2} $w_1, e_1\subset X - X\cap Y$ and $w_2, e_2\subset Y - X\cap Y$ (or vice versa). In this case, $\check{q}_{X\cap Y}$ is a homeomorphism and $X_1\cap Y_1 = q(X\cap Y)\sqcup q(e_1)\sqcup q(w_1)$.
\item $e_1\subset X - X\cap Y$, $e_2\subset Y-X\cap Y$ (or vice versa) and $w_1 = w_2\in X\cap Y$. In this case, $\check{q}_{X\cap Y}$ is a homeomorphism and $X_1\cap Y_1 = q(X\cap Y) \sqcup q(e_1)$.
\end{enumerate}
\end{lemma}

\begin{proof}
If $q$ is an $X$-fold and a $Y$-fold, then $q\mid X\cap Y$ is also a fold and so $e_1, e_2\subset X\cap Y$ and $q(X\cap Y) = X_1\cap Y_1)$. If $q$ is an $X$-fold, but not a $Y$-fold, then $e_1, e_2\subset X$ and either $w_1 = w_2$ or $w_1\neq w_2$ and $w_1, w_2$ are not both in $Y$ by \cref{not_fold}. In particular, $q(X\cap Y) = X_1\cap Y_1$. The case in which $q$ is a $Y$-fold, but not an $X$-fold is handled similarly. This covers the first case.

Now suppose that $q$ is neither an $X$-fold or a $Y$-fold. Then certainly $e_1\subset X - X\cap Y$ and $e_2\subset Y - X\cap Y$ (or vice versa). Since $q$ is ordinary, $\check{q}_X, \check{q}_Y$ must be homeomorphism. \cref{not_fold} then implies that either $w_1, w_2\in X\cap Y$. Thus, the two remaining cases correspond to whether $w_1 \neq w_2$ or $w_1= w_2$. In the first case we have that $X_1\cap Y_1$ consists of $q(X\cap Y)$ and $q(e_1) \sqcup q(w_1)$ while in the second case we have that $X_1\cap Y_1$ consists of $q(X\cap Y)$ and $q(e_1)$.
\end{proof}

\begin{lemma}
\label{ordinary_fold_rr}
If $q$ is an ordinary fold, then $(Z_1, Y_1, X_1)$ is $\psi$-bi-invariant for $H$ and
\begin{align*}
\rr(Z_1, Y_1)&\leqslant \rr(Z, X),\\
\rr(Z_1, X_1)&\leqslant \rr(Z, Y).
\end{align*}
Moreover if $\rr(Z, X), \rr(Z, Y)<\infty$, then
\[
\rr(Z_1, Y_1) + \rr(Z_1, X_1)<\rr(Z, Y) + \rr(Z, X)
\]
if and only if $e_1\cup e_2$ is a bigon not in $X\cap Y$.
\end{lemma}

\begin{proof}
By definition of ordinary folds, we have $Z^{\#} = Z_1^{\#}$ and $X^{\#} = X_1^{\#}, Y^{\#} = Y_1^{\#}$. In the first two cases of \cref{ordinary_fold}, we also have $(X\cap Y)^{\#} = (X_1\cap Y_1)^{\#}$ and so in these cases $(Z_1, Y_1, X_1)$ is $\psi$-bi-invariant since $(Z, Y, X)$ was. Now suppose that we are in the third case of \cref{ordinary_fold}; that is, $e_1\subset X - X\cap Y$, $e_2\subset Y - X\cap Y$ (or vice versa) and $w_1 = w_2$. Let $\alpha, \beta\colon I\to X\cap Y$ be paths connecting the basepoint with $v, w_1 = w_2$ respectively. Then denote by $\delta_X = f_Z\circ(\alpha*e_1*\overline{\beta})$ and $\delta_Y = f_Z\circ(\alpha*e_2*\overline{\beta})$. We have that $[\delta_X] = [\delta_Y]$ and $(X_1\cap Y_1)^{\#} = \langle (X\cap Y)^{\#}, [\delta_X]\rangle$. Since $[\delta_X]\in X^{\#}$, $[\delta_Y]\in Y^{\#}$, we see that $\psi([\delta_X])\in Y^{\#}$ and $\psi^{-1}([\delta_X])\in X^{\#}$ and so $(Z_1, Y_1, X_1)$ is also $\psi$-bi-invariant in this case.

\cref{rr_inequality} gives us the desired bounds on reduced ranks. 

Now suppose that $\rr(Z, X), \rr(Z, Y)<\infty$. Using the equality $\rr(Z, X) = -\chi(Z, X)$, the only way that we can have $\rr(Z_1, X_1)<\rr(Z, X)$ is if $Z_1 - X_1$ is obtained from $Z - X$ by removing an edge. This only can happen if $e_1, e_2\subset Z - X$ and $w_1 = w_2$ or if $e_1\subset X - X\cap Y$, $e_2\subset Y$ (or vice versa). By \cref{ordinary_fold}, the latter case can only happen if $w_1 = w_2$ and so $q$ folds a bigon not in $X\cap Y$. The case in which $\rr(Z_1, Y_1)<\rr(Z, Y)$ is handled similarly.
\end{proof}

\subsection{Single exceptional folds}

When $q$ is a single exceptional fold, the new graph triple $(Z_1, Y_1, X_1)$ may no longer be $\psi$-bi-invariant. This will be fixed in the subsequent lemmas. First, we identify precisely when a fold is a single exceptional fold.

\begin{lemma}
\label{single_fold}
If $q$ is a single exceptional fold, then, up to exchanging $e_1$ and $e_2$, one of the following holds:
\begin{enumerate}
\item $q$ is a single $X$-exceptional fold. In this case, either
\begin{enumerate}
\item\label{itm:case1a} $\check{q}_X$ is a fold and so $\check{q}_{X\cap Y}$ identifies a pair of vertices and $X_1\cap Y_1 = q(X\cap Y)$, or
\item\label{itm:case1b} $\check{q}_X$ is a homeomorphism and so $e_1\subset X - X\cap Y$, $e_2\subset Y - X\cap Y$, $w_1\in X\cap Y$, $w_2\in Y-X\cap Y$ and $X_1\cap Y_1 = q(X\cap Y) \sqcup q(e_1)$.
\end{enumerate}
\item $q$ is a single $Y$-exceptional fold. In this case, either
\begin{enumerate}
\item\label{itm:case2a} $\check{q}_Y$ is a fold and so $\check{q}_{X\cap Y}$ identifies a pair of vertices and $X_1\cap Y_1 = q(X\cap Y)$, or
\item\label{itm:case2b} $\check{q}_Y$ is a homeomorphism and so $e_1\subset X - X\cap Y$, $e_2\subset Y - X\cap Y$, $w_1\in X - X\cap Y$, $w_2\in X\cap Y$ and $X_1\cap Y_1 = q(X\cap Y) \sqcup q(e_1)$.
\end{enumerate}
\end{enumerate}
\end{lemma}

\begin{proof}
The four subcases are mutually exclusive and cover all possibilities by definition of single exceptional folds.

By \cref{not_fold}, if $\check{q}_Y$ is not a fold or a homeomorphism, then $w_1\neq w_2$, $w_1, w_2\in Y$ and $e_1\subset X - X\cap Y$ or $e_2\subset X - X\cap Y$. 

If additionally, $\check{q}_X$ is a fold, then we must have $e_1, e_2\subset X$. In this case, $\check{q}_{X\cap Y}$ identifies a pair of vertices and $X_1\cap Y_1 = q(X\cap Y)$. This is Case \eqref{itm:case1a}. 

If instead $\check{q}_X$ is a homeomorphism, then $e_1\subset X - X\cap Y$ and $e_2\subset Y - X\cap Y$ (or vice versa) and $X_1\cap Y_1$ is obtained from $q(X\cap Y)$ by adjoining the edge $q(e_1)$. This is Case \eqref{itm:case1b}.

Cases \eqref{itm:case2a} and \eqref{itm:case2b} are handled similarly.
\end{proof}

In the next lemma we show how to fix $\psi$-bi-invariance of the triple $(Z_1, Y_1, X_1)$ when $q$ is a (single) $X$-exceptional fold. The case of a $Y$-exceptional fold is handled in the exact same way by exchanging the roles of $X$ and $Y$ and replacing $\psi$ with $\psi^{-1}$.

\begin{lemma}
\label{lem:case1}
Suppose that $q$ is a single $X$-exceptional fold.
\begin{itemize}
\item If $q$ is as in Case \eqref{itm:case1a} from \cref{single_fold}, let $\alpha, \beta\colon I \to X\cap Y$ be an immersed path connecting $v_Z$ with $w_1, w_2$ respectively and denote by $\delta = [(f_Z\circ\alpha)*(f_Z\circ \overline{\beta})] \in \F$. 
\item If $q$ is as in Case (\ref{itm:case1b}) from \cref{single_fold}, let $\alpha, \beta\colon I\to X\cap Y$ be immersed paths connecting $v_Z$ with $w_1$ and $v$ respectively and denote by $\delta = [(f_Z\circ\alpha)*(f_Z\circ\overline{e_1})*(f_Z\circ\overline{\beta})]\in \F$. 
\end{itemize}
If we put
\begin{align*}
X_2 &= \begin{cases}
		X_1 \quad &\text{ if $\psi^{-1}(\delta)\in X^{\#}$}\\
		X_1 \vee \Gamma(\psi^{-1}(\delta)) \quad &\text{ otherwise}
	\end{cases}\\
Y_2 &= Y_1\\
Z_2 &= X_2 \cup Y_2
\end{align*}
then $(Z_2, Y_2, X_2)$ is a $\psi$-bi-invariant graph triple for $H$ and
\begin{align*}
\rr(Z_2, X_2) &\leqslant \rr(Z, X)\\
\rr(Z_2, Y_2) &\leqslant \rr(Z, Y).
\end{align*}
If $\psi^{-1}(\delta)\in X^{\#}$ and $\rr(Z, Y)<\infty$, then $\rr(Z_2, Y_2)<\rr(Z, Y)$. 
\end{lemma}

\begin{proof}
By \cref{single_fold} we have that $(X_1\cap Y_1)^{\#} = \langle (X\cap Y)^{\#}, \delta\rangle$. Hence, since $X_2\cap Y_2 = X_1\cap Y_1$, we have $\psi^{-1}((X_2\cap Y_2)^{\#})\leqslant\psi^{-1}(Y_2^{\#})\leqslant X_2^{\#}$. Thus, $(X_2, X_2\cap Y_2)$ and $(Z_2, Y_2)$ are $\psi^{-1}$-invariant. Moreover, since $\check{q}_X$ is a fold or a homeomorphism and $(Z, X)$ was $\psi$-invariant, it also follows that $\psi((X_2\cap Y_2)^{\#})\leqslant\psi(X_2^{\#})\leqslant Y_2^{\#}$ and so $(Y_2, X_2\cap Y_2)$ and $(Z_2, X_2)$ are $\psi$-invariant. Hence, $(Z_2, Y_2, X_2)$ is $\psi$-bi-invariant.

By \cref{rr_inequality}, we have that $\rr(Z_2, X_2) = \rr(Z_1, X_1)\leqslant \rr(Z, X)$. Since $q$ is an $X$-exceptional fold, we have that $Y_1$ is obtained from $Y$ by identifying two vertices and so $Z_1 - Y_1$ is obtained from $Z - Y$ by removing an edge. Thus, $\rr(Z_2, Y_2)\leqslant \rr(Z_1, Y_1) + 1 = \rr(Z, X)$. By construction, the inequality is strict precisely when $\psi^{-1}(\delta)\in X^{\#}$.
\end{proof}

\subsection{Double exceptional folds}

Finally we analyse the double exceptional folds case. Just as in the case of a single exceptional fold, the graph triple $(Z_1, Y_1, X_1)$ may no longer be $\psi$-bi-invariant.

\begin{lemma}
\label{double_fold}
If $q$ is a double exceptional fold, then $w_1\neq w_2$, $e_1\subset X - Y\cap X$, $e_2\subset Y - X\cap Y$ and $w_1, w_2\in X\cap Y$ (or vice versa). Moreover, $\check{q}_{X\cap Y}$ is not a fold or a homeomorphism and $X_1\cap Y_1 - q(X\cap Y) = q(e_1)$.
\end{lemma}

\begin{proof}
This follows from applying \cref{not_fold} with $A = X, Y$.
\end{proof}

\begin{lemma}
\label{double_fold_fix}
Suppose that $q$ is a double exceptional fold. Let $\alpha, \beta\colon I\to X\cap Y$ be immersed paths connecting $v_Z$ with $w_1$ and $w_2$ respectively and let $\delta = [(f_Z\circ\alpha)*(f_Z\circ\beta)]\in Z^{\#}\leqslant \F$. If 
\begin{align*}
X_2 &= \begin{cases}
		X_1 \quad &\text{ if $\psi^{-1}(\delta)\in X^{\#}$}\\
		X_1 \vee \Gamma(\psi^{-1}(\delta)) \quad &\text{ otherwise}
	\end{cases}\\
Y_2 &= \begin{cases}
		Y_1 \quad &\text{ if $\psi(\delta)\in Y^{\#}$}\\
		Y_1 \vee \Gamma(\psi(\delta)) \quad &\text{ otherwise}
	\end{cases}\\
Z_2 &= X_2 \cup Y_2
\end{align*}
then $(Z_2, X_2, Y_2)$ is $\psi$-bi-invariant for $H$ and
\begin{align*}
\rr(Z_2, Y_2)&\leqslant \rr(Z, Y),\\
\rr(Z_2, X_2)&\leqslant \rr(Z, X).
\end{align*}

If $\psi^{-1}(\delta)\in X^{\#}$ and $\rr(Z, Y)<\infty$, then $\rr(Z_2, Y_2)<\rr(Z, Y)$. 

If $\psi(\delta)\in Y^{\#}$ and $\rr(Z, X)<\infty$, then $\rr(Z_2, X_2)<\rr(Z, X)$.
\end{lemma}

\begin{proof}
The proof is very similar to the proof of \cref{lem:case1} and so we leave the details to the reader.
\end{proof}

\subsection{Folding and adding loops if necessary}

We may now define folding and adding loops if necessary for graph triples.

\begin{definition}
\label{def:folding2}
If $(Z, Y, X)$ is a $\psi$-bi-invariant graph triple, $q\colon Z\to Z_1$ is a fold and $(Z_2, Y_2, X_2)$ is
\begin{itemize}
\item the $\psi$-bi-invariant graph triple $(Z_1, Y_1, X_1)$ if $q$ is an ordinary fold (see \cref{ordinary_fold_rr}),
\item the $\psi$-bi-invariant graph triple obtained in \cref{lem:case1} if $q$ is a single exceptional fold (after possibly exchanging the roles of $X$ and $Y$),
\item the $\psi$-bi-invariant graph triple obtained in \cref{double_fold_fix} if $q$ is a double exceptional fold,
\end{itemize}
then we say that $(Z_2, Y_2, X_2)$ is obtained from $(Z, Y, X)$ by \emph{folding and adding loops if necessary}.

If $\check{q}_X$ is a fold, then we say that $(Z_2, Y_2, X_2)$ is obtained from $(Z, Y, X)$ by \emph{folding $X$ and adding a loop if necessary} (two loops cannot be added). If $\check{q}_Y$ is a fold, then we say that $(Z_2, Y_2, X_2)$ is obtained from $(Z, Y, X)$ by \emph{folding $Y$ and adding a loop if necessary}.
\end{definition}

\begin{remark}
\label{FH_remark}
If the graph triple $(Z_2, Y_2, X_2)$ is obtained from $(Z, Y, X)$ by folding $Y$ and adding a loop if necessary, then the graph pair $(Y_2, X_2\cap Y_2)$ is obtained from $(Y, X\cap Y)$ by folding and adding a loop if necessary in Feighn--Handel's sense \cite{FH99}. Similarly, if the graph triple $(Z_2, Y_2, X_2)$  is obtained from $(Z, Y, X)$ by folding $X$ and adding a loop if necessary, then the graph pair $(X_2, X_2\cap Y_2)$ is obtained from $(X, X\cap Y)$ by folding and adding a loop if necessary in Feighn--Handel's sense (in the $\psi^{-1}$ direction).
\end{remark}

We summarise the most important properties of graph triples obtained by folding and adding loops if necessary that we have established below.

\begin{proposition}
\label{summary}
If $q\colon Z\to Z_1$ is a fold and if $(Z_2, Y_2, X_2)$ is obtained from $(Z_1, Y_1, X_1)$ by adding loops if necessary, then $(Z_2, Y_2, X_2)$ is $\psi$-bi-invariant graph triple for $H$ and the following holds:
\begin{enumerate}
\item We have
\begin{align*}
\rr(Z_2, Y_2) &\leqslant \rr(Z, Y)\\
\rr(Z_2, X_2) &\leqslant \rr(Z, X)
\end{align*}
In particular, if $(Z, Y, X)$ is minimal, then so is $(Z_2, Y_2, X_2)$.
\item If one of the following holds:
\begin{itemize}
\item $q$ is a subgraph fold with $e_1\cup e_2$ a bigon not contained in $X\cap Y$.
\item $q$ is a single exceptional fold and no loop is added.
\item $q$ is a double exceptional fold and at most one loop is added.
\end{itemize}
Then
\[
\rr(Z_2, Y_2) + \rr(Z_2, X_2) < \rr(Z, Y) +\rr(Z, X)
\]
if $\rr(Z, Y), \rr(Z, X)<\infty$.
\item If $Z$ is finite and $q$ is not as in \eqref{itm:ordinary_case2} from \cref{ordinary_fold}, then either the number of vertices in $X_2\cap Y_2$ is strictly less than the number of vertices in $X\cap Y$, or the number of edges in $Z_2$ is strictly less than the number of edges in $Z$.
\end{enumerate}
\end{proposition}

\subsection{Tightening $X$, $Y$ and $Z$}

We are now ready to describe our variation of the Feighn--Handel tightening procedure for graph triples.

\begin{definition}{[\emph{Tightening $X$, tightening $Y$}]}
Let $(Z, Y, X)$ be a $\psi$-bi-invariant graph triple. If $X$ is tight then do nothing. If $X\cap Y$ is not tight, then fold a pair of edges in $X\cap Y$. If $X\cap Y$ is tight, but $X$ is not tight, then fold a pair of edges in $X$ and add a loop if necessary. Repeat until we obtain a new $\psi$-bi-invariant graph triple $(\check{Z}, \check{Y}, \check{X})$ with $\check{X}$ tight. We say $(\check{Z}, \check{Y}, \check{X})$ is obtained from $(Z, Y, X)$ by \emph{tightening $X$}. We may similarly define the \emph{tightening $Y$} procedure.
\end{definition}

\begin{remark}
\label{remark:FH}
If $(Z, Y, X)$ is a $\psi$-bi-invariant graph triple and $(\check{Z}, \check{Y}, \check{X})$ is obtained from $(Z, Y, X)$ by the tightening $Y$ procedure, then the graph pair $(\check{Y}, \check{X}\cap \check{Y})$ is obtained from the graph pair $(Y, X\cap Y)$ by tightening in the sense of Feighn--Handel \cite{FH99}, see \cref{FH_remark}. The same may be said about the graph pair $(\check{X}, \check{X}\cap \check{Y})$, except that we replace the automorphism $\psi$ with the inverse $\psi^{-1}$.
\end{remark}

One could derive the following from \cref{FH_remark} and the remarks in \cite[Definition 4.8]{FH99}. We prove it directly using \cref{summary}.

\begin{lemma}
\label{tightening_X_or_Y}
If $(Z, Y, X)$ is a finite $\psi$-bi-invariant graph triple for $H$, then the tightening of $X$ procedure terminates in a $\psi$-bi-invariant graph triple $(\check{Z}, \check{Y}, \check{X})$ for $H$ with $\check{X}$ tight. Similarly, the tightening of $Y$ procedure terminates in a $\psi$-bi-invariant graph triple $(\check{Z}, \check{Y}, \check{X})$ for $H$ with $\check{Y}$ tight. In both cases, we have $\rr(\check{Z}, \check{X}) \leqslant \rr(Z, X)$, $\rr(\check{Z}, \check{Y}) \leqslant \rr(Z, Y)$.
\end{lemma}

\begin{proof}
The fact that after each folding and adding loops if necessary we obtain a new finite $\psi$-bi-invariant graph triple for $H$ is by \cref{summary}. Letting
\[
\begin{tikzcd}
{(Z_0, Y_0, X_0)} \arrow[r, "q_0"] & {(Z_1, Y_1, X_1)} \arrow[r, "q_1"] & \ldots \arrow[r, "q_{k-1}"] & {(Z_k, Y_k, X_k)} \arrow[r, "q_k"] & {\ldots}
\end{tikzcd}
\]
be the sequence of folds and adding loops if necessary from the procedure, where $(Z, Y, X) = (Z_0, Y_0, X_0)$. By \cref{summary}, we have that $q_i(X_i\cap Y_i) = X_{i+1}\cap Y_{i+1}$ for all $i$. If $q_i$ is an exceptional fold, then this implies that the number of vertices in $X_{i+1}\cap Y_{i+1}$ is strictly less than the number of vertices in $X_i\cap Y_i$. Hence, only finitely many exceptional folds are performed. If $q_i$ is not an exceptional fold, then no loop is added and so the number of edges in $Z_{i+1}$ is strictly less than the number of edges in $Z_i$. Thus, the procedure must terminate after finitely many steps. 
\end{proof}

\begin{definition}{[\emph{Tightening $X$ and $Y$}]}
Let $(Z, Y, X)$ be a $\psi$-bi-invariant graph triple. If $X$ and $Y$ are tight, then do nothing. If $X$ is not tight, then tighten $X$ or if $Y$ is not tight, then tighten $Y$. Repeat until we obtain a new $\psi$-bi-invariant triple $(\check{Z}, \check{Y}, \check{X})$ with both $\check{Y}$ and $\check{X}$ tight.
\end{definition}

\begin{lemma}
\label{tightening_XY}
If $(Z, Y, X)$ is a finite $\psi$-bi-invariant graph triple for $H$, then the tightening of $X$ and $Y$ procedure terminates in a $\psi$-bi-invariant graph triple $(\check{Z}, \check{Y}, \check{X})$ for $H$ with $\check{Y}$ and $\check{X}$ tight and with $\rr(\check{Z}, \check{X}) \leqslant \rr(Z, X)$, $\rr(\check{Z}, \check{Y}) \leqslant \rr(Z, Y)$.
\end{lemma}

\begin{proof}
The proof is almost identical to that of \cref{tightening_X_or_Y}. We simply need to note that throughout the procedure, only finitely many exceptional folds are performed (using \cref{summary}) and so after all such fold is performed, each subsequent fold is ordinary and strictly decreases the number of edges.
\end{proof}

Combining \cref{remark:FH}, \cite[Proposition 5.4]{FH99} and \cref{tightening_XY} we obtain the following.

\begin{proposition}
\label{prop:injective}
Let $(Z, Y, X)$ be a finite $\psi$-bi-invariant graph triple for $H$ with $\pi_1(X\cap Y, v_Z)\to \pi_1(R, v_R)$ injective, but $\pi_1(X, v_Z)\to \pi_1(R, v_R)$ not injective. If $(\check{Z}, \check{Y}, \check{X})$ is obtained from $(Z, Y, X)$ by the tightening $X$ procedure, then $\rr(\check{Z}, \check{Y}) < \rr(Z, Y)$ and $\rr(\check{Z}, \check{X})\leqslant\rr(Z, X)$.
\end{proposition}

There is one final natural tightening procedure (which we shall not use) that one can define.

\begin{definition}{[\emph{Tightening $Z$}]}
Let $(Z, Y, X)$ be a $\psi$-bi-invariant graph triple. If $Z$ is tight then do nothing. If $X$ is not tight, then tighten $X$ or if $Y$ is not tight, then tighten $Y$. If $X$ and $Y$ are tight, but $Z$ is not tight, then fold and add loops if necessary. Repeat until we obtain a tight $\psi$-bi-invariant triple $(\check{Z}, \check{Y}, \check{X})$.
\end{definition}

We do not know whether there is an analogue of \cref{tightening_XY} for the tightening $Z$ procedure. 

\begin{question}
Does the tightening $Z$ procedure always terminate after finitely many steps when given a finite $\psi$-bi-invariant graph triple?
\end{question}

In \cite{Li25b} direct limits of maps of graph pairs are considered. Using the same ideas, one could take an appropriate direct limit of graph triples to obtain a tight $\psi$-bi-invariant graph triple for a given finitely generated subgroup. It will still have finite relative rank, but may not be finite itself.

\section{Main theorems}

\subsection{Presentations of free-by-cyclic groups}
\label{sec:presentations}

Using \cref{prop:injective} we may obtain a refinement of \cref{FHmain} for graph triples.

\begin{theorem}
\label{FHmain2}
Let $(Z, Y, X)$ be a finite $\psi$-bi-invariant graph triple for $H$ with $(f_Z)_*$ injective and let $C, D\leqslant Z^{\#}$ so that $Z^{\#} = X^{\#}*C = D*Y^{\#} = D*(X\cap Y)^{\#}*C$. The following are equivalent:
\begin{enumerate}
\item\label{itm:min2} The triple $(Z, Y, X)$ is minimal.
\item\label{itm:min3} The $\psi$-invariant pair $(Z, X)$ is minimal.
\item\label{itm:min4} The $\psi^{-1}$-invariant pair $(Z, Y)$ is minimal.
\item\label{itm:inj2} The map
\[
\theta_n\colon\pi_1\left(\bigvee_{i = 0}^n\Gamma(\psi^{-i}(D))\vee(X\cap Y)\vee \bigvee_{i=0}^n\Gamma(\psi^i(C)), v_Z\right) \to \pi_1(R, v_R)
\]
is injective for all $n\geqslant 0$.
\item\label{itm:pres2} We have
\[
H\isom \langle Z^{\#}, t \mid t^{-1}xt = \phi(x), \forall x\in X^{\#}\rangle.
\]
where $\phi\colon X^{\#} \to Y^{\#}$ is the isomorphism given by $\phi = \psi\mid_{X^{\#}}$.
\end{enumerate}
In particular, if any of the above hold, then $-\chi(H) = \rr(Z, X) = \rr(Z, Y)$.
\end{theorem}

\begin{proof}
The fact that \cref{itm:min3} and \cref{itm:min4} are both equivalent to \cref{itm:pres2} is \cref{FHmain} combined with \cref{iso}. Using the fact that:
\begin{align*}
D*(X\cap Y)^{\#} &= X^{\#} = \langle (X\cap Y)^{\#}, \psi^{-1}((X\cap Y)^{\#})\\
(X\cap Y)^{\#}*C &= Y^{\#} = \langle (X\cap Y)^{\#}, \psi((X\cap Y)^{\#})\rangle
\end{align*}
and that $\psi(D*(X\cap Y)^{\#}) = (X\cap Y)^{\#}*C$ by \cref{iso}, by induction on $n$ we see that
\begin{align*}
\pi_1&\left(\bigvee_{i = 0}^n\Gamma(\psi^{-i}(D))\vee(X\cap Y)\vee \bigvee_{i=0}^n\Gamma(\psi^i(C)), v_X\right)\\
&= \psi^{-n}\left(\pi_1\left((X\cap Y)\vee \bigvee_{i=0}^{2n}\Gamma(\psi^i(C)), v_X\right)\right),
\end{align*}
and so \cref{FHmain} also implies that \cref{itm:pres2} is equivalent to \cref{itm:inj2}. Finally, repeated application of \cref{prop:injective} shows that \cref{itm:min2} implies \cref{itm:inj2}. Finally, since $-\chi(H) = \rr(Z, X) = \rr(Z, Y)$ is a group invariant, we see that \cref{itm:pres2} implies \cref{itm:min2} and we are done.
\end{proof}

\subsection{Proof of \cref{thm:main1}}

It is clear that $G$ is finitely generated if $\F$ admits such a free product decomposition so we only focus on proving the other direction.

Let $R$ be a rose graph with $\pi_1(R, v_R) \isom \F$. Let $(Z, Y, X)$ be a finite minimal $\psi$-bi-invariant graph triple for $\F\rtimes_{\psi}\Z$, with $X$ and $Y$ tight and with $f_Z\colon (Z, v_Z)\to (R, v_R)$ the associated graph map. Applying \cref{FHmain2} we obtain that
\[
\F = \left(\Asterisk_{i= 0}^{-\infty}\psi^i(D)\right)*E*\left(\Asterisk_{i=0}^{\infty}\psi^i(C)\right).
\]
where $E = (X\cap Y)^{\#}$. For each $j\geqslant 0$, denote by
\[
\F_j = \left(\Asterisk_{i= 0}^{-j}\psi^i(D)\right)*E*\left(\Asterisk_{i=0}^{\infty}\psi^i(C)\right).
\]
For each $j\geqslant 0$, we have 
\[
\psi^{-j}(\F_0) = \F_{j} = \psi^{-j}(E)*\left(\Asterisk_{i=-j}^{\infty}\psi^i(C)\right).
\]
Since $\F_j$ is a free factor of $\F$, we see that
\[
\bigast_{i\geqslant -j} \psi^i(C)
\]
is a free factor of $\F$ for all $j\geqslant 0$ and so
\[
\bigast_{i\in \Z} \psi^i(C)
\]
is a free factor of $\F$. Letting $A\leqslant \F$ be a free factor such that $\F = A*\left(\bigast_{i\in \Z} \psi^i(C)\right)$, it suffices to show that $A$ is finitely generated to complete the proof as we put $C_0 = C$. 

Let $A'$ be any finitely generated free factor of $A$. Then there is some $j\geqslant 0$ such that $A'\leqslant \F_{j}$. Since $A'*\bigast_{i\geqslant -j} \psi^i(C)$ is a free factor of $\F$, it is also a free factor of $\F_{j}$. Since $\F_{j}/\normal{\bigast_{i\geqslant -j} \psi^i(C)}\isom \psi^{-j}(E)$, and since the projection map $\F_{j} \to \psi^{-j}(E)$ sends the free factor $A'$ injectively to a free factor of $\psi^{-j}(E)$, it follows that $\rk(A')\leqslant \rk(E)$. Since $A'$ was an arbitrary finitely generated free factor of $A$, we see that $\rk(A)\leqslant \rk(E)$ and so $A$ is finitely generated as claimed.

\begin{remark}
Note that, using the notation of \cref{thm:main1}, $\chi(\F\rtimes_{\psi}\Z) = -\rk(C_0)$. In particular, the rank of $C_0$ is an invariant of the free-by-cyclic group. The rank of the group $A$ is certainly not an invariant since there are groups with several epimorphisms to $\Z$ with finitely generated free kernels of different ranks. However, we do not know whether $\rk(A)$ is an invariant of the automorphism $\psi$. When $C_0 = 1$, then this is trivially the case as $\F = A$.
\end{remark}

\subsection{The geometry of free-by-cyclic groups}

Combining \cref{thm:main1} with the main result from \cite{Li25b} (see also \cite[Theorem 5.1]{Li25b} and the discussion preceding it) we may also state a geometric refinement. The reader is invited to consult Hruska's article \cite{Hr10} for more information on relatively hyperbolic groups and relatively quasi-convex subgroups.

\begin{theorem}
\label{thm:rel_hyp}
Let $\F$ be a free group, let $\psi\in \aut(\F)$ be an automorphism and suppose that $G = \F\rtimes_{\psi}\Z$ is finitely generated. Then $\F$ admits a free product decomposition
\[
\F = A_1*\ldots*A_n*B*\left(\Asterisk_{i\in \Z}C_i\right)
\]
with the following properties:
\begin{enumerate}
\item $A_1, \ldots, A_n, B, C_0$ are each finitely generated and $C_i = \psi^i(C_0)$ for each $i\in \Z$.
\item There is a constant $k$ so that for any finitely generated subgroup $H\leqslant \F$, if $\psi^m(H)$ intersects a conjugate of $H$ non-trivially for some $m\geqslant k$, then $H$ is conjugate into some $A_i$.
\item There is an element $\sigma\in \sym(n)$ and elements $f_1, \ldots, f_n\in \F$ so that $\psi(A_i)^{f_i} = A_{\sigma(i)}$.
\item For each $i$, if $m_i\geqslant 1$ is minimal so that $\sigma^{m_i}(i) = i$, then the group
\[
H_i = \langle A_i, tf_itf_{\sigma(i)}\ldots tf_{\sigma^{m_i-1}(i)}\rangle
\]
is \{fg free\}-by-cyclic.
\item If $\mathcal{P}$ is a collection of conjugacy class representatives of the subgroups $H_i$, then any subgroup of $G$ isomorphic to a \{fg free\}-by-cyclic group is conjugate into a unique $P\in \mathcal{P}$.
\end{enumerate}
 Moreover, $(G, \mathcal{P})$ is relatively hyperbolic and locally relatively quasi-convex.
\end{theorem}

\subsection{Proof of \cref{thm:main2}}

Let $\F = A*\left(\Asterisk_{i\in \Z}C_i\right)$ be the free product decomposition from \cref{thm:main1}. Let $D$ be a finitely generated free group of rank at least $\rk(C_0)$ and let $\phi\in \aut(D)$ be an atoroidal automorphism so that $D\rtimes_{\phi}\Z$ is hyperbolic by a result of Brinkmann \cite[Theorem 1.2]{Br00} (such an automorphism always exists when the rank of $D$ is at least three). By work of Hagen--Wise \cite[Theorem A]{HW15} combined with Agol's result \cite[Theorem 1.1]{Ag13}, after passing to a finite index subgroup, we may assume that $D\rtimes_{\phi}\Z$ is also compact special. Now choose any free generating set $\mathcal{D}\subset D$ for a free factor of $D$ of rank $\rk(C_0)$. For some $n>\!\!\!>1$, if $\mathcal{D}_n$ denotes the set obtained from $\mathcal{D}$ by elevating each generator to the power $n$, the subgroup $\langle \mathcal{D}_n, t^n\rangle$ is a quasi-convex free group of rank $|\mathcal{D}_n| + 1$ by a result of Arzhantseva \cite[Theorem 1]{Ar01}. Hence, after replacing $\phi$ with $\phi^n$, we may assume that $\langle \mathcal{D}_n, t\rangle$ is a quasi-convex free subgroup and hence that 
\[
\langle \mathcal{D}_n, t\rangle\cap D \isom \Asterisk_{i\in \Z}\langle\phi^i(\mathcal{D}_n)\rangle.
\]
Since $\langle \mathcal{D}_n, t\rangle$ is a quasi-convex subgroup of the hyperbolic and compact special group $D\rtimes_{\theta}\Z$, after possibly passing to a finite index subgroup of $D\rtimes_{\phi}\Z$ containing $\langle \mathcal{D}_n, t\rangle$, we may assume that there is also a retraction $\rho\colon D\rtimes_{\phi}\Z\to \langle \mathcal{D}_n, t\rangle$ by a result of Haglund--Wise (see \cite[Theorem 7.3]{HW08} and its proof).

Choose any isomorphism $\lambda_0\colon C_0\to \langle \mathcal{D}_n\rangle$ and define $\lambda_i\colon C_i\to \langle \phi^i(\mathcal{D}_n)\rangle$ as $\lambda_i = \phi^i\circ\lambda_0\circ\psi^{-i}\mid C_i$. This then extends to an isomorphism 
\[
\lambda\colon \Asterisk_{i\in \Z}C_i\to \Asterisk_{i\in \Z}\langle\phi^i(\mathcal{D}_n)\rangle
\]
so that $\lambda\circ\psi = \phi\circ\lambda$ when restricted to $\Asterisk_{i\in \Z}C_i$.

Now let $\iota\colon \F\to F = A*D$ be the homomorphism given by extending the inclusion of $A$ and $\lambda$. Note that $\iota(\F)\cap D = \Asterisk_{i\in \Z}C_i$. By construction, the homomorphism $\theta\colon A*D\to A*D$ given by $\theta\mid A = \iota\circ\psi\mid A$ and $\theta\mid D = \phi$ is an automorphism which restricts to $\iota\circ\psi$ on $\iota(\F)$. Hence, $\F\rtimes_{\psi}\Z$ embeds into $(A*D)\rtimes_{\theta}\Z$ via the homomorphism extending $\iota$ and the identity map on $\Z$. Moreover, the map $F\rtimes_{\theta}\Z\to \F\rtimes_{\psi}\Z$ given by $\rho$ on $D\rtimes_{\theta}\Z$ and by the identity on $\F\rtimes_{\psi}\Z$ is a retraction.

We now need the following lemma. We thank Ashot Minasyan for pointing this out to us.

\begin{lemma}
\label{lem:retract}
If $G$ is a retract of a torsion-free hyperbolic group $H$, then $G$ is malnormal in $H$.
\end{lemma}

\begin{proof}
Let $\rho\colon H\to G$ be a retraction and let $h\in H$ and suppose that there is some non-trivial $g\in G^h\cap G$. Since $\rho\mid G = \id$, we see that $g = \rho(g)$ and $hgh^{-1} = \rho(hgh^{-1}) = \rho(h)g\rho(h^{-1})$. This implies that $h^{-1}\rho(h)\in C_H(g)$. Since $H$ is torsion-free hyperbolic, centralisers of elements are infinite cyclic. Thus, some power of $h^{-1}\rho(h)$ is equal to a power of $g$. Since $h^{-1}\rho(h)\in \ker(\rho)$ and $H$ is torsion-free, this implies that $h^{-1}\rho(h) = 1$. Thus, $h\in G$ and so $G$ is malnormal as claimed.
\end{proof}

By \cref{lem:retract}, we see that $\langle \mathcal{D}_n, t\rangle$ is malnormal in $D\rtimes_{\phi}\Z$. Since $D\rtimes_{\phi}\Z$ is hyperbolic, we may use a result of Bowditch \cite[Theorem 7.11]{Bo12} to conclude that $(D\rtimes_{\phi}\Z, \{\langle \mathcal{D}_n, t\rangle\})$ is relatively hyperbolic. Then by a result of Dahmani \cite[Theorem 0.1]{Da03}, we see that $(F\rtimes_{\theta}\Z, \{\F\rtimes_{\psi}\Z\})$ is relatively hyperbolic.

\begin{remark}
Replacing the use of \cref{thm:main1} with \cite[Corollary 4.7]{Li25b} and making other minor modifications to the proof of \cref{thm:main2}, one may obtain the analogous result for mapping tori of free groups. This was already obtained by Chong--Wise in \cite[Theorem 1.1]{CW24}.
\end{remark}

\bibliographystyle{amsalpha}
\bibliography{bibliography}

\end{document}